%% file: main.tex
\documentclass{article}
\usepackage{amsmath}
\usepackage{amsfonts}
\usepackage{latexsym}
\usepackage{mathtools}
\usepackage{amsthm}
\usepackage{amssymb}
\usepackage{tikz}
\numberwithin{equation}{section}
\newcommand{\cat}{\text{cat}}
\newcommand{\ccat}{\text{ccat}}
\newcommand{\Int}{\mathrm{Int}}
\newcommand{\R}{\mathbb{R}}
\newcommand{\Z}{\mathbb{Z}}

\newcommand{\grad}{\text{grad}}
\newcommand{\TC}{\text{TC}}
\newtheorem{theorem}{Theorem}[section]
\newtheorem*{theorem*}{Theorem}
\newtheorem{lemma}[theorem]{Lemma}
\newtheorem{definition}[theorem]{Definition}
\newtheorem{proposition}[theorem]{Proposition}
\newtheorem{corollary}[theorem]{Corollary}
\newtheorem{remark}[theorem]{Remark}
\newtheorem{example}[theorem]{Example}
\newtheorem*{acknowledgments*}{Acknowledgments}
\usepackage[hidelinks]{hyperref}
\usepackage{docmute}
\title{Topological complexity for closed 1-forms}
\author{Kenji Fukushi}

\date{}
\begin{document}
\maketitle

\begin{abstract}
\sloppy
Michael Farber introduced a generalization of the Lusternik--Schnirelmann
category for closed 1-forms.  In this paper, we introduce and
study a corresponding version of topological complexity for closed 1-forms.
We establish analogues of the basic properties of ordinary topological
complexity, including inequalities relating this invariant to the
Lusternik--Schnirelmann category for closed 1-forms.  We also prove a
closed 1-form analogue of the navigation-function method: upper bounds for our invariant can be obtained from the dynamical properties of
gradient flows of closed 1-forms.
\end{abstract}

\input{Intro.tex}

\input{definition.tex}

\input{cupproduct.tex}

\input{productineq_sub.tex}

\input{navigation_finalfantasy.tex}

\section*{Acknowledgments}
The author wishes to thank his supervisor Tsuyoshi Kato for constant encouragement and helpful suggestions.

Department of Mathematics, Kyoto University,
Kyoto, Japan

\textit{Email address}: \texttt{fukushi.kenji.85n@st.kyoto-u.ac.jp}

\end{document}

%% file: Intro.tex
\sloppy

\section{Introduction}

Topological complexity is a homotopy invariant arising
from the motion planning problem in robotics \cite{Far01}. 
It measures the minimal number of continuous local rules required to choose a path between any two points of a given space.

\begin{definition}[(Unreduced) topological complexity, {\cite{Far01}}]\label{def:TC}
Let \(PX\) be the space of continuous paths in a topological space \(X\), and define
\[
\pi \colon PX \to X \times X,\qquad
\pi(\gamma) = (\gamma(0),\gamma(1)).
\]
The topological complexity \(\TC(X)\) is the minimal integer \(k\) such that
\(X \times X\) admits an open covering
\[
X \times X = G_1 \cup \cdots \cup G_k
\]
for which each \(G_i\) admits a continuous section
\[
s_i \colon G_i \to PX
\]
of \(\pi\).
\end{definition}

The Lusternik--Schnirelmann category \(\cat(X)\) is also a homotopy invariant \cite{CLOT03}.
In the unreduced convention, it is defined as the minimal integer \(k\ge 1\)
such that \(X\) admits an open covering
\[
X=U_1\cup\cdots\cup U_k
\]
where each \(U_i\) is contractible in \(X\).
These invariants are closely related.
For example, one has the standard inequalities \cite{Far01}
\[
\cat(X) \le \TC(X) \le \cat(X\times X) \le 2\cat(X)-1.
\]

The Lusternik--Schnirelmann category associated with a closed 1-form was
introduced by Farber \cite{Far03,Far02}.
This invariant is motivated by Novikov's generalization of Morse theory to closed 1-forms \cite{Nov01, Nov02}. 
It reflects both topological and dynamical properties of closed 1-forms.

Closed 1-forms on cell complexes, as well as their line integrals along paths,
are understood in the sense of \cite[Section~10.2]{Far02}.
However, all explicit computations in this paper are carried out for ordinary smooth closed 1-forms on closed manifolds.

\begin{definition}[\(N\)-movable subset]
Let \(X\) be a finite cell complex, \(\xi\in H^1(X;\R)\), and let \(\omega\) be a closed 1-form representing \(\xi\). A subset \(U\subset X\) is called \(N\)-movable with respect to \(\omega\) if there exists a homotopy \[ H\colon U\times [0,1]\to X \] such that \(H(x,0)=x\) and \[ \int_{\gamma_x}\omega \le -N \] for every \(x\in U\), where \(\gamma_x(t)=H(x,t)\).
\end{definition}

\begin{definition}[Lusternik--Schnirelmann category for closed 1-forms,
{\cite[Definition~10.7]{Far02}}]
Let \(X\) be a finite cell complex, let \(\xi\in H^1(X;\mathbb R)\),
and let \(\omega\) be a closed 1-form representing \(\xi\).
The Lusternik--Schnirelmann category of \(X\) with respect to
\(\xi\), denoted by \(\cat(X,\xi)\), is the minimal integer \(k\)
such that, for every integer \(N>0\), there exists an open covering
\[
X = U \cup U_1 \cup \cdots \cup U_k
\]
where \(U\) is \(N\)-movable with respect to \(\omega\), and each \(U_i\)
is contractible in \(X\).
\end{definition}

Classically, the Lusternik--Schnirelmann category gives a lower bound for
the number of critical points of a smooth function on a closed manifold \cite{Tak01}.
One of the main features of the Lusternik--Schnirelmann category of closed
1-forms is its relation to the dynamics of gradient flows of closed 1-forms \cite{FS02}.  
We recall the following homoclinic-cycle theorem.

\begin{theorem}[{\cite[Theorem 10.15]{Far02}}]\label{thm:Farber-homoclinic-cycle}
Let \(M\) be a closed manifold, and let $\omega$ be a closed 1-form on $M$.
If \(\omega\) has fewer zeros than $\cat(M,[\omega])$,
then every gradient flow for \(\omega\) has a homoclinic cycle.
\end{theorem}
A homoclinic cycle is a cyclic sequence of flow lines joining zeros of
\(\omega\), where each trajectory connects one zero to the next and the
last trajectory returns to the first zero.
Thus, \(\cat(M,\xi)\) detects dynamical phenomena that do not occur in the ordinary gradient-flow theory of functions.

In this paper, for a finite cell complex \(X\) and a cohomology class
\(\xi\in H^1(X;\R)\), we introduce an analogue of topological complexity
associated with \(\xi\), denoted \(\TC(X,\xi)\). 
Roughly speaking, \(\TC(X,\xi)\) measures the motion-planning problem for
pairs \((x,y)\in X\times X\) whose target point \(y\) cannot be moved by a
homotopy with arbitrarily large negative \(\omega\)-integral.

For example, such a homotopy may arise from the flow \(\varphi_t\) of a
vector field.  If, along the trajectory of \(y\),
\[
\int_0^T \omega\bigl(\dot\varphi_t(y)\bigr)\,dt
\]
tends to \(-\infty\) as \(T\to\infty\), then \(y\) is movable
with respect to \(\omega\).  The invariant \(\TC(X,\xi)\) ignores this
movable part and measures the remaining motion-planning complexity.

We establish analogues of the standard inequalities relating ordinary
topological complexity to Lusternik--Schnirelmann category, as well as
cohomological lower bounds expressed in terms of cup products with local
coefficients.

A new feature of this invariant is that upper bounds for \(\TC(X,\xi)\)
may be derived from the dynamical properties of gradient flows of
closed 1-forms.  In the classical setting, upper bounds for \(\TC(X)\)
can be obtained from so-called navigation functions \cite[Section~4.4]{Far04}, namely suitable
Morse--Bott functions whose gradient flows organize local motion planners.
We prove an analogue of this method in which the navigation function is
replaced by a closed 1-form.  Thus \(\TC(X,\xi)\) connects motion planning, Lusternik--Schnirelmann category, and the dynamics of
closed 1-forms.

%% file: definition.tex
\sloppy

\section{Definition of topological complexity of closed 1-forms and some properties}

\subsection{Definition}

Throughout this paper, all spaces are assumed to be finite cell complexes.
Let \(X\) be a finite cell complex, and let
\[
\pi\colon PX\to X\times X,\qquad
\pi(\gamma)=(\gamma(0),\gamma(1))
\]
be the endpoint fibration. We use the unreduced convention for topological
complexity.

Let \(\omega\) be a closed 1-form on \(X\), and let
\[
p_1,p_2\colon X\times X\to X
\]
be the two projections. Choose a closed 1-form \(\Omega\) on \(X\times X\)
such that
\[
[\Omega]
=
p_2^*[\omega]-p_1^*[\omega]
\in H^1(X\times X;\R).
\]

\begin{definition}[Topological complexity for a closed 1-form]
\label{def:TComega}
The topological complexity associated with \(\omega\), denoted by
\(\TC(X,[\omega])\), is the smallest integer \(k\) such that, for every
integer \(N>0\), there exists an open covering
\[
X\times X
=
F^{(N)}\cup F_1^{(N)}\cup\cdots\cup F_k^{(N)}
\]
satisfying the following conditions:
\begin{itemize}
\item[(1)] For each \(i=1,\dots,k\), there exists a continuous map
\[
s_i^{(N)}\colon F_i^{(N)}\to PX
\]
such that
\[
\pi\circ s_i^{(N)}
=
\mathrm{id}_{F_i^{(N)}}.
\]

\item[(2)] There exist a constant \(C>0\), independent of \(N\), and a
homotopy
\[
H_N\colon F^{(N)}\times[0,1]\to X\times X
\]
such that
\[
H_N(z,0)=z
\]
for every \(z\in F^{(N)}\). Writing
\[
\gamma_N^z(t)=H_N(z,t),
\]
one has
\[
\int_{\gamma_N^z|_{[0,t]}}\Omega\le C
\]
for every \(z\in F^{(N)}\) and every \(t\in[0,1]\), and
\[
\int_{\gamma_N^z}\Omega\le -N
\]
for every \(z\in F^{(N)}\).
\end{itemize}
\end{definition}

\begin{remark}
The set \(F^{(N)}\) consists of points of \(X\times X\) which can be
moved arbitrarily far in the negative direction of the closed 1-form
\(\Omega\). Such points are not counted among the local motion-planning
domains in the definition of \(\TC(X,[\omega])\).
\end{remark}

\begin{remark}
One may take
\[
\Omega=p_2^*\omega-p_1^*\omega.
\]
In this case, writing
\[
H_N((x,y),t)
=
\bigl(\alpha_N^{x,y}(t),\beta_N^{x,y}(t)\bigr),
\]
one has
\[
\int_{\gamma_N^{(x,y)}|_{[0,t]}}\Omega
=
\int_{\beta_N^{x,y}|_{[0,t]}}\omega
-
\int_{\alpha_N^{x,y}|_{[0,t]}}\omega.
\]
Thus Definition~\ref{def:TComega} agrees with the movable condition written
in terms of the two coordinates.
\end{remark}

\begin{proposition}[Well-definedness]
\label{prop:representative-independence}
The invariant \(\TC(X,[\omega])\) is independent of the choice of the
closed 1-form \(\Omega\) satisfying
\[
[\Omega]
=
p_2^*[\omega]-p_1^*[\omega].
\]
Moreover, it depends only on the cohomology class \([\omega]\).
\end{proposition}

\begin{proof}
Let \(\Omega'\) be another closed 1-form on \(X\times X\) satisfying
\[
[\Omega']
=
p_2^*[\omega]-p_1^*[\omega].
\]
Then
\[
\Omega'=\Omega+df
\]
for some continuous function
\[
f\colon X\times X\to\R.
\]
Since \(X\times X\) is compact, there exists \(K>0\) such that
\[
|f|\le K.
\]

Suppose that
\[
X\times X
=
F\cup F_1\cup\cdots\cup F_k
\]
satisfies the defining conditions with respect to \(\Omega\). For every
\(N>0\), let
\[
H_N\colon F\times[0,1]\to X\times X
\]
be the corresponding homotopy, and write
\[
\gamma_N^z(t)=H_N(z,t).
\]
For every \(t\in[0,1]\), one has
\[
\int_{\gamma_N^z|_{[0,t]}}df
=
f(\gamma_N^z(t))-f(z),
\]
and hence
\[
\left|
\int_{\gamma_N^z|_{[0,t]}}df
\right|
\le 2K.
\]
It follows that
\[
\int_{\gamma_N^z|_{[0,t]}}\Omega'
\le C+2K.
\]
At the final time,
\[
\int_{\gamma_N^z}\Omega'
\le -N+2K.
\]
Given \(N'>0\), choose an integer \(N\) such that
\[
N\ge N'+2K.
\]
Then
\[
\int_{\gamma_N^z}\Omega'\le -N'.
\]
Thus the same open covering satisfies the defining conditions with
respect to \(\Omega'\), after replacing \(C\) by \(C+2K\). Therefore the
number obtained using \(\Omega'\) is no larger than the number obtained
using \(\Omega\). By symmetry, the two numbers are equal.

Now let \(\omega'\) be another closed 1-form on \(X\) such that
\[
[\omega']=[\omega].
\]
Then
\[
p_2^*[\omega']-p_1^*[\omega']
=
p_2^*[\omega]-p_1^*[\omega].
\]
Hence the same closed 1-form \(\Omega\) may be used in the definition for
both \(\omega\) and \(\omega'\). Therefore
\[
\TC(X,[\omega'])=\TC(X,[\omega]).
\]
\end{proof}

For
\[
\xi\in H^1(X;\R),
\]
choose any closed 1-form \(\omega\) representing \(\xi\), and define
\[
\TC(X,\xi):=\TC(X,[\omega]).
\]
By Proposition~\ref{prop:representative-independence}, this definition is
independent of the choice of \(\omega\).

If \([\omega]=0\), then
\[
[\Omega]=0.
\]
Hence
\[
\Omega=df
\]
for some continuous function
\[
f\colon X\times X\to\R.
\]
Since \(X\times X\) is compact, the integral of \(\Omega\) along any path
is uniformly bounded. Therefore the movable condition cannot hold for
arbitrarily large \(N\) unless the movable set is empty. Consequently,
\[
\TC(X,[\omega])=\TC(X).
\]
Equivalently,
\[
\TC(X,0)=\TC(X).
\]

\subsection{Homotopy invariance}

\begin{proposition}
Let \(X\) and \(Y\) be finite connected cell complexes, and let
\[
\varphi\colon X\to Y
\]
be a homotopy equivalence. Let \(\omega\) be a closed 1-form on \(Y\).
Then
\[
\TC(X,[\varphi^*\omega])=\TC(Y,[\omega]).
\]
Equivalently, for every
\[
\xi\in H^1(Y;\R),
\]
one has
\[
\TC(X,\varphi^*\xi)=\TC(Y,\xi).
\]
\end{proposition}

\begin{proof}
Choose a closed 1-form \(\Omega\) on \(Y\times Y\) such that
\[
[\Omega]
=
p_2^*[\omega]-p_1^*[\omega].
\]
Then
\[
\Omega_X=(\varphi\times\varphi)^*\Omega
\]
satisfies
\[
[\Omega_X]
=
p_2^*[\varphi^*\omega]-p_1^*[\varphi^*\omega]
\]
on \(X\times X\).

Let
\[
\psi\colon Y\to X
\]
be a homotopy inverse of \(\varphi\). We prove
\[
\TC(X,[\varphi^*\omega])\le \TC(Y,[\omega]).
\]
The reverse inequality follows by interchanging \(X\) and \(Y\).

Let
\[
Y\times Y
=
F\cup F_1\cup\cdots\cup F_k
\]
be an open covering satisfying the defining conditions for
\(\TC(Y,[\omega])\), using the representative \(\Omega\). Put
\[
F'=(\varphi\times\varphi)^{-1}(F),
\qquad
F_i'=(\varphi\times\varphi)^{-1}(F_i).
\]
Then
\[
X\times X
=
F'\cup F_1'\cup\cdots\cup F_k'.
\]

The local motion planners on the sets \(F_i'\) are obtained in the usual
way. Namely, one concatenates a fixed homotopy from
\(\mathrm{id}_X\) to \(\psi\varphi\), the path obtained by applying
\(\psi\) to the local motion planner on \(F_i\), and the reverse of the
same fixed homotopy at the second endpoint.

It remains to verify the movable condition on \(F'\). Let
\[
H_N\colon F\times[0,1]\to Y\times Y
\]
be the movable homotopy with respect to \(\Omega\). Using a fixed
homotopy from \(\mathrm{id}_X\) to \(\psi\varphi\), followed by
\[
(\psi\times\psi)\circ H_N\circ
\bigl((\varphi\times\varphi)\times\mathrm{id}_{[0,1]}\bigr),
\]
one obtains a homotopy on \(F'\) starting at the inclusion.

Since
\[
(\varphi\psi\times\varphi\psi)^*\Omega
\]
and \(\Omega\) represent the same cohomology class on \(Y\times Y\),
their difference is exact. The integral contributed by this exact
difference is uniformly bounded. The integrals along the fixed
homotopies are also uniformly bounded because the spaces involved are
compact. Consequently, there exists a constant \(C'>0\), independent of
\(N\), such that the resulting homotopy satisfies the controlled
condition with respect to \(\Omega_X\), and its final
\(\Omega_X\)-integral is at most
\[
-N+C'.
\]
Replacing \(N\) by \(N+C'\) gives the required integral threshold
\(-N\). Therefore
\[
\TC(X,[\varphi^*\omega])\le \TC(Y,[\omega]).
\]
The reverse inequality is analogous.
\end{proof}

\subsection{Closed 1-forms without zeros}

We give a simple computable example.

\begin{proposition}
Let \(M\) be a closed manifold, and let \(\omega\) be a closed 1-form on
\(M\) with no zeros. Then
\[
\TC(M,[\omega])=0.
\]
\end{proposition}

\begin{proof}
Take
\[
\Omega=p_2^*\omega-p_1^*\omega
\]
on \(M\times M\). Since \(\omega\) has no zeros, \(\Omega\) has no zeros.

Fix a Riemannian metric \(G\) on \(M\times M\), and let
\[
V=-\Omega^{\sharp_G}
\]
be the negative metric dual of \(\Omega\). Then
\[
\Omega(V)=-|V|_G^2.
\]
Since \(\Omega\) has no zeros and \(M\times M\) is compact, there exists
\(a>0\) such that
\[
\Omega(V)\le -a
\]
on \(M\times M\).

Let \(\phi_t\) be the flow generated by \(V\). Given \(N>0\), choose
\(T_N>0\) such that
\[
aT_N\ge N.
\]
Put
\[
F=M\times M
\]
and define
\[
H_N(z,s)=\phi_{sT_N}(z).
\]
Then
\[
\int_{H_N(z,\cdot)|_{[0,s]}}\Omega
=
\int_0^{sT_N}\Omega(V)(\phi_t(z))\,dt
\le 0
\]
for every \(s\in[0,1]\), and
\[
\int_{H_N(z,\cdot)}\Omega
=
\int_0^{T_N}\Omega(V)(\phi_t(z))\,dt
\le -aT_N
\le -N.
\]
Thus the whole space \(M\times M\) is the movable part in
Definition~\ref{def:TComega}. Hence
\[
\TC(M,[\omega])=0.
\]
\end{proof}

For example, if \(d\theta\) denotes the angular form on \(S^1\), then
\[
\TC(S^1,[d\theta])=0.
\]

\subsection{\texorpdfstring{A conditional lower bound for \(\TC(X,[\omega])\)}
{A conditional lower bound for TC(X,[omega])}}

We give a conditional comparison with the Lusternik--Schnirelmann
category of a closed 1-form.

\begin{proposition}
\label{prop:cat-lower-bound}
Let \(X\) be a finite connected cell complex, and let \(\omega\) be a
closed 1-form on \(X\). In the definition of \(\TC(X,[\omega])\), take
\[
\Omega=p_2^*\omega-p_1^*\omega.
\]
Assume that the movable part can be chosen so that there exist a point
\(x_0\in X\) and a constant \(C>0\), independent of \(N\), satisfying
\[
\int_{\alpha_N^{x_0,y}}\omega\le C
\]
for every \(y\) such that \((x_0,y)\) belongs to the movable part. Then
\[
\cat(X,[\omega])\le \TC(X,[\omega]).
\]
\end{proposition}

\begin{proof}
Assume that
\[
\TC(X,[\omega])=k.
\]
For every \(N>0\), there exists an open covering
\[
X\times X
=
F\cup F_1\cup\cdots\cup F_k
\]
satisfying the defining conditions with respect to
\[
\Omega=p_2^*\omega-p_1^*\omega.
\]

Let \(N'>0\), and choose
\[
N=N'+C.
\]
Write the movable homotopy on \(F\) as
\[
H_N((x,y),t)
=
\bigl(\alpha_N^{x,y}(t),\beta_N^{x,y}(t)\bigr).
\]
Define
\[
U=\{y\in X\mid (x_0,y)\in F\},
\qquad
U_i=\{y\in X\mid (x_0,y)\in F_i\}.
\]
Then
\[
X=U\cup U_1\cup\cdots\cup U_k.
\]

For every \(y\in U\), the movable condition gives
\[
\int_{\beta_N^{x_0,y}}\omega
-
\int_{\alpha_N^{x_0,y}}\omega
=
\int_{\gamma_N^{(x_0,y)}}\Omega
\le -N.
\]
By the additional hypothesis,
\[
\int_{\alpha_N^{x_0,y}}\omega\le C.
\]
Hence
\[
\int_{\beta_N^{x_0,y}}\omega
\le C-N
=
-N'.
\]
Thus \(U\) is \(N'\)-movable with respect to \(\omega\).

Each \(U_i\) is contractible in \(X\). Indeed, if
\[
s_i\colon F_i\to PX
\]
is a local section, then the path
\[
s_i(x_0,y)
\]
joins \(x_0\) to \(y\) and depends continuously on \(y\). Reversing
these paths gives a contraction of \(U_i\) to \(x_0\) in \(X\).

Therefore, for every \(N'>0\), the space \(X\) admits an open covering
\[
X=U\cup U_1\cup\cdots\cup U_k
\]
where \(U\) is \(N'\)-movable with respect to \(\omega\), and each
\(U_i\) is contractible in \(X\). Hence
\[
\cat(X,[\omega])
\le k
=
\TC(X,[\omega]).
\]
\end{proof}

\begin{remark}
The additional hypothesis in
Proposition~\ref{prop:cat-lower-bound} is a boundedness condition on the
first coordinate. A zero of a negative gradient flow of \(\omega\)
gives a typical case of this condition, since the first coordinate is
fixed and the negative displacement is detected by the second
coordinate.
\end{remark}

\subsection{\texorpdfstring{An upper bound for \(\TC(X,[\omega])\)}
{An upper bound for TC(X,[omega])}}

We now prove an upper bound for \(\TC(X,[\omega])\) in terms of the
controlled Lusternik--Schnirelmann category of closed 1-forms
\(\ccat(X,[\omega])\) in the sense of Farber--Sch\"utz
\cite[Definition~9.4]{FS01}.

\begin{proposition}
\label{prop:TC-upper-by-controlled-cat}
Let \(X\) be a finite connected cell complex, and let \(\omega\) be a
closed 1-form on \(X\). Set
\[
\eta
=
p_2^*[\omega]-p_1^*[\omega]
\in H^1(X\times X;\R).
\]
Then
\[
\TC(X,[\omega])
\le
\ccat(X\times X,\eta).
\]
\end{proposition}

\begin{proof}
Let
\[
k=\ccat(X\times X,\eta),
\]
and choose a closed 1-form \(\Omega\) on \(X\times X\) satisfying
\[
[\Omega]=\eta.
\]
By the definition of \(\ccat\), for every \(N>0\), there exists an open
covering
\[
X\times X
=
F\cup G_1\cup\cdots\cup G_k
\]
such that \(F\) satisfies the controlled movable condition with respect
to \(\Omega\), and each \(G_i\) is contractible in \(X\times X\).

Since each \(G_i\) is contractible in \(X\times X\), the endpoint
fibration
\[
\pi\colon PX\to X\times X
\]
admits a continuous section over \(G_i\). Hence each \(G_i\) is a local
motion-planning domain.

The controlled movable condition for \(F\) with respect to \(\Omega\)
is exactly the movable condition in
Definition~\ref{def:TComega}. Therefore the above covering is admissible
for \(\TC(X,[\omega])\), and hence
\[
\TC(X,[\omega])
\le k
=
\ccat(X\times X,\eta).
\]
\end{proof}

\begin{corollary}
\label{cor:TC-upper-by-cat}
Let \(X\) be a finite connected cell complex, and let \(\omega\) be a
closed 1-form on \(X\). Then
\[
\TC(X,[\omega])
\le
\ccat(X,-[\omega])+\ccat(X,[\omega])-1.
\]
\end{corollary}

\begin{proof}
By Proposition~\ref{prop:TC-upper-by-controlled-cat}, one has
\[
\TC(X,[\omega])
\le
\ccat\bigl(
X\times X,
p_2^*[\omega]-p_1^*[\omega]
\bigr).
\]
Since
\[
p_2^*[\omega]-p_1^*[\omega]
=
p_1^*(-[\omega])+p_2^*[\omega],
\]
the product inequality for the controlled
Lusternik--Schnirelmann category of closed 1-forms
\cite[Theorem~9]{FS01} gives
\[
\ccat\bigl(
X\times X,
p_2^*[\omega]-p_1^*[\omega]
\bigr)
\le
\ccat(X,-[\omega])+\ccat(X,[\omega])-1.
\]
Therefore
\[
\TC(X,[\omega])
\le
\ccat(X,-[\omega])+\ccat(X,[\omega])-1.
\]
\end{proof}

\begin{example}
\label{exa:Sigma2-TC-one}
Let \(\omega_\Sigma\) be a closed 1-form on \(\Sigma_2\) such that
\[
0\ne[\omega_\Sigma]\in H^1(\Sigma_2;\R).
\]
By \cite{FS01}, one has
\[
\ccat(\Sigma_2,[\omega_\Sigma])
=
\ccat(\Sigma_2,-[\omega_\Sigma])
=
1.
\]
Moreover, the additional hypothesis of
Proposition~\ref{prop:cat-lower-bound} is satisfied in this case by
using a negative gradient flow on \(\Sigma_2\) \cite{Sch01}.

By Proposition~\ref{prop:cat-lower-bound}, one has
\[
1
=
\cat(\Sigma_2,[\omega_\Sigma])
\le
\TC(\Sigma_2,[\omega_\Sigma]).
\]
On the other hand, by Corollary~\ref{cor:TC-upper-by-cat},
\[
\begin{aligned}
\TC(\Sigma_2,[\omega_\Sigma])
&\le
\ccat(\Sigma_2,-[\omega_\Sigma])
+
\ccat(\Sigma_2,[\omega_\Sigma])
-1\\
&=
1+1-1\\
&=1.
\end{aligned}
\]
Therefore
\[
\TC(\Sigma_2,[\omega_\Sigma])=1.
\]
\end{example}

\begin{corollary}
\label{cor:Sigma2-times-S2-upper}
Let
\[
X=\Sigma_2\times S^2,
\]
and let
\[
\omega=\pi_\Sigma^*\omega_\Sigma,
\]
where \(\omega_\Sigma\) is a closed 1-form on \(\Sigma_2\) satisfying
\[
0\ne[\omega_\Sigma]\in H^1(\Sigma_2;\R).
\]
Then
\[
\TC(X,[\omega])\le 3.
\]
\end{corollary}

\begin{proof}
By the product inequality for the controlled
Lusternik--Schnirelmann category of closed 1-forms,
\[
\begin{aligned}
\ccat(X,[\omega])
&=
\ccat\bigl(
\Sigma_2\times S^2,
\pi_\Sigma^*[\omega_\Sigma]
\bigr)\\
&\le
\ccat(\Sigma_2,[\omega_\Sigma])
+
\ccat(S^2,0)
-1.
\end{aligned}
\]
Since
\[
\ccat(\Sigma_2,\pm[\omega_\Sigma])=1
\]
and
\[
\ccat(S^2,0)=\cat(S^2)=2,
\]
we obtain
\[
\ccat(X,[\omega])\le 2.
\]
Similarly,
\[
\ccat(X,-[\omega])\le 2.
\]
By Corollary~\ref{cor:TC-upper-by-cat},
\[
\TC(X,[\omega])
\le
\ccat(X,-[\omega])+\ccat(X,[\omega])-1
\le
2+2-1
=
3.
\]
Therefore
\[
\TC(\Sigma_2\times S^2,[\pi_\Sigma^*\omega_\Sigma])
\le 3.
\]
\end{proof}

%% file: cupproduct.tex
\sloppy

\section{\texorpdfstring{A cup-product lower bound for \(\TC(X,\xi)\)}
{A cup-product lower bound for TC(X,xi)}}

In this section we prove a cohomological lower bound for
\(\TC(X,\xi)\) using local coefficients.  The argument is based on the
cup-product estimate for the Lusternik--Schnirelmann category of closed
1-forms due to Farber and Sch\"utz~\cite{FS01}.

Let \(X\) be a finite connected polyhedron, let
\(\xi\in H^1(X;\R)\), and set \(Y=X\times X\).  For the projections
\(p_1,p_2\colon Y\to X\), put
\[
\eta=p_2^*\xi-p_1^*\xi\in H^1(Y;\R).
\]
Denote the diagonal map by
\[
\Delta\colon X\to Y.
\]

We recall the convention on transcendental local systems from
\cite[Definition~6.1]{FS01}.  Given a finite connected cell complex \(Z\)
and a class \(\theta\in H^1(Z;\R)\), denote by
\[
\operatorname{Ker}(\theta)\subset H_1(Z;\Z)
\]
the subgroup of homology classes on which \(\theta\) evaluates trivially,
and set
\[
H_\theta=H_1(Z;\Z)/\operatorname{Ker}(\theta).
\]
Writing
\[
r=\operatorname{rank} H_\theta,
\]
we have
\[
H_\theta\cong\Z^r.
\]
The set \(V_\theta\) consists of all complex flat line bundles over \(Z\)
whose monodromy is trivial on \(\operatorname{Ker}(\theta)\).  Equivalently,
\[
V_\theta=\operatorname{Hom}(H_\theta,\mathbb C^*)
\cong(\mathbb C^*)^r.
\]
For \(L\in V_\theta\), denote by
\[
\operatorname{Mon}_L\colon \Z[H_\theta]\to\mathbb C
\]
the monodromy homomorphism.  The flat line bundle \(L\) is called
\emph{transcendental} if \(\operatorname{Mon}_L\) is injective.

We recall the notion of a neighborhood of infinity
\cite[Definition~2.2]{FS01}.  Let
\[
p\colon \widetilde Y\to Y
\]
be the covering corresponding to \(\operatorname{Ker}(\eta)\), let
\(H_\eta\) be its group of deck transformations, and let
\[
A\colon \widetilde Y\to H_\eta\otimes\mathbb R
\]
be an Abel--Jacobi map~\cite[Proposition2.1]{FS01}.  A subset \(U\subset\widetilde Y\) is called a
\emph{neighborhood of infinity with respect to \(\eta\)} if there exists
\(c\in\mathbb R\) such that
\[
\bigl\{x\in\widetilde Y
\mid \eta_{\mathbb R}(A(x))>c\bigr\}\subset U.
\]
The property that \(U\) is a neighborhood of infinity is independent of
the choice of the Abel--Jacobi map \(A\) and of the representative of
\(\eta\).

A homology class
\[
z\in H_q(\widetilde Y;\mathbb C)
\]
is said to be \emph{movable to infinity with respect to \(\eta\)} if,
for every neighborhood of infinity
\[
U\subset\widetilde Y
\]
with respect to \(\eta\), the class \(z\) can be represented by a
singular cycle whose support is contained in \(U\).

\begin{theorem}\label{thm:TC-local-system-cup}
Let \(X\) be a finite connected polyhedron, let
\[
0\ne\xi\in H^1(X;\R),
\]
and let \(k\ge0\).  Set
\[
Y=X\times X
\]
and
\[
\eta=p_2^*\xi-p_1^*\xi\in H^1(Y;\R).
\]
Suppose that there exist a transcendental flat complex line bundle
\[
\mathcal L\in V_\eta,
\]
a class
\[
v_0\in H^{d_0}(Y;\mathcal L),
\]
and positive-degree classes
\[
u_1,\dots,u_k\in
\ker\Bigl(
\Delta^*\colon H^{>0}(Y;\mathbb C)
\to H^{>0}(X;\mathbb C)
\Bigr)
\]
such that
\[
v_0\smile u_1\smile\cdots\smile u_k\ne0
\]
in \(H^*(Y;\mathcal L)\).  Then
\[
\TC(X,\xi)\ge k+1.
\]
\end{theorem}

\begin{proof}
Let \(\theta\) be a continuous closed
1-form on \(Y\) representing \(\eta\) and 
\[
p\colon \widetilde Y\to Y
\]
be the covering corresponding to \(\operatorname{Ker}(\eta)\).  We use the proof of     \cite[Theorem~5]{FS01}.
The only modification is that the null-homotopic open sets in the
Lusternik--Schnirelmann category argument are replaced by local-planner
sets.

Assume, for contradiction, that
\[
\TC(X,\xi)\le k.
\]
Since
\[
v_0\smile u_1\smile\cdots\smile u_k\ne 0
\]
in \(H^*(Y;\mathcal L)\), and since \(\mathcal L\) is transcendental,
\cite[Proposition~6.5]{FS01} gives a homology class
\[
z\in H_*(\widetilde Y;\mathbb C)
\]
such that, with
\[
u=u_1\smile\cdots\smile u_k,
\]
one has
\[
\langle v_0\smile u,p_*(z)\rangle\ne0.
\]
Equivalently,
\[
\langle v_0,p_*(p^*u\cap z)\rangle\ne0.
\]

Choose a singular cycle \(c\) representing \(z\), and choose a compact
subpolyhedron
\[
K\subset \widetilde Y
\]
which contains the support of \(c\).
Since \(p\) is the covering associated with \(\operatorname{Ker}(\eta)\), we have
\[
p^*\theta=df
\]
for some continuous function \(f\colon\widetilde Y\to\mathbb R\). 

As in the proof of \cite[Theorem~5]{FS01}, choose a neighborhood of
infinity with respact to \(-\eta\)
\[
U\subset\widetilde Y
\]
with the following property: any homology class in
\[
H_*(K;\mathbb C)
\]
which is homologous in \(\widetilde Y\) to a cycle supported in \(U\) is
movable to infinity with respect to \(-\eta\).

Choose numbers \(c_0<a<b\) such that
\[
f(K)\subset[a,b],
\qquad
f^{-1}((-\infty,c_0))\subset U,
\]
and choose
\[
N>b-c_0.
\]

By the definition of \(\TC(X,\xi)\), there exists an open cover
\[
Y=F\cup G_1\cup\cdots\cup G_k
\]
such that \(F\) is \(N\)-movable with respect to \(\eta\), and each \(G_i\)
admits a continuous local section
\[
s_i\colon G_i\to PX
\]
of the endpoint fibration
\[
\pi\colon PX\to X\times X.
\]

For each \(i\), the local section \(s_i\) homotopes \(G_i\) into the diagonal.
Thus \(u_i|_{G_i}=0\), and hence \(u_i\) may be regarded as a relative class
in \(H^*(Y,G_i;\mathbb C)\).  Taking the relative cup product, we see that
\[
u=u_1\smile\cdots\smile u_k
\]
vanishes on \(G_1\cup\cdots\cup G_k\).

Choose open sets \(A_0,B_0\) and compact subpolyhedra \(A,B\) of \(Y\)
such that
\[
A_0\cup B_0=Y,
\qquad
\overline{A_0}\subset A\subset F,
\qquad
\overline{B_0}\subset B\subset G_1\cup\cdots\cup G_k.
\]
Since
\[
u|_{G_1\cup\cdots\cup G_k}=0,
\]
the class \(u\) may be represented by a singular cocycle \(g\) which
vanishes on every singular simplex contained in \(B_0\).

Choose a singular cycle \(c\) representing \(z\).  After sufficiently
subdividing \(c\), we may assume that the projection to \(Y\) of every
simplex occurring in \(c\) is contained in either \(A_0\) or \(B_0\).
Put
\[
z_0=p^*u\cap z.
\]
Then \(z_0\) is represented by the cycle
\[
p^*g\cap c
\]
(\cite[chapter~5]{Spa01}).
Since \(g\) vanishes on every simplex contained in \(B_0\), the projection
of the support of \(p^*g\cap c\) is contained in \(A_0\subset F\).

We now use the \(N\)-movability of \(F\).  
Since the cycle
\(p^*g\cap c\) is supported in \(K\) and projects into \(F\), the
\(N\)-movable homotopy of \(F\) lifts to a homotopy which decreases \(f\)
by more than \(N\).  Since \(f\le b\) on \(K\) and \(N>b-c_0\), the
resulting cycle is supported in
\[
f^{-1}((-\infty,c_0))\subset U.
\]
Hence \(z_0\) is homologous in \(\widetilde Y\) to a cycle supported in
\(U\).
By the defining property of \(U\), it follows that \(z_0\) is
movable to infinity with respect to \(-\eta\).

On the other hand,
since \(\mathcal L\) is transcendental, 
\[
\langle v_0,p_*(z_0)\rangle\ne0
\]
and \cite[Proposition~6.4]{FS01}, \cite[Theorem~4]{FS01} imply that \(z_0\) is not movable to infinity
with respect to \(-\eta\), which is a contradiction.
Hence
\[
\TC(X,\xi)>k.
\]
\end{proof}

\begin{proposition}\label{prop:Sigma2-times-S2-cup-lower}
Let
\[
X=\Sigma_2\times S^2
\]
and let
\[
\xi=\pi_{\Sigma}^*\xi_{\Sigma}\in H^1(X;\R),
\]
where
\[
0\ne \xi_{\Sigma}\in H^1(\Sigma_2;\R).
\]
Then
\[
\TC(X,\xi)=3.
\]
\end{proposition}

\begin{proof}
Let
\[
z\in H^2(S^2;\mathbb C)
\]
be the generator, and put
\[
\alpha=\pi_{S^2}^*z\in H^2(X;\mathbb C).
\]
Then
\[
u=p_2^*\alpha-p_1^*\alpha
\]
is a diagonal zero-divisor, since
\[
\Delta^*u=\alpha-\alpha=0.
\]
Moreover,
\[
u^2
=
(p_2^*\alpha-p_1^*\alpha)^2
=
-2\,p_1^*\alpha\smile p_2^*\alpha
\ne 0.
\]

Choose a transcendental flat complex line bundle
\[
L_\Sigma\in V_{\xi_\Sigma}
\]
on \(\Sigma_2\), and put
\[
L=\pi_\Sigma^*L_\Sigma\in V_\xi.
\]
For a nontrivial rank-one local system on \(\Sigma_2\), one has
\[
H^0(\Sigma_2;L_\Sigma)=0,
\qquad
H^2(\Sigma_2;L_\Sigma)=0,
\]
and hence, by the Euler characteristic,
\[
\dim H^1(\Sigma_2;L_\Sigma)=2.
\]
The same holds for \(L_\Sigma^*\).  By the K\"unneth formula, there exist
nonzero classes
\[
a\in H^1(X;L^*),
\qquad
b\in H^1(X;L).
\]
Set
\[
\mathcal L=p_1^*L^*\otimes p_2^*L.
\]
Then
\[
\mathcal L\in V_{p_2^*\xi-p_1^*\xi}.
\]
The homomorphism
\[
H_1(X;\mathbb Z)\oplus H_1(X;\mathbb Z)\longrightarrow H_1(X;\mathbb Z),
\qquad
(a,b)\longmapsto b-a,
\]
induces an isomorphism
\[
H_{p_2^*\xi-p_1^*\xi}\cong H_\xi.
\]
Under this isomorphism, the monodromy of
\(\mathcal L=p_1^*L^*\otimes p_2^*L\) is the monodromy of \(L\).
Hence, if \(L\) is transcendental, then \(\mathcal L\) is transcendental.

Define
\[
v_0=p_1^*a\smile p_2^*b
\in H^2(X\times X;\mathcal L).
\]
Then
\[
v_0\smile u\smile u
=
-2\,p_1^*(a\smile\alpha)\smile p_2^*(b\smile\alpha)
\ne 0.
\]
Therefore Theorem~\ref{thm:TC-local-system-cup} with \(k=2\) implies
\[
\TC(X,\xi)\ge 3.
\]
Combining this lower bound with Corollary~\ref{cor:Sigma2-times-S2-upper}, we obtain
\[
\TC(X,\xi)=3.
\]
\end{proof}

For comparison, one has
\[
\cat(X,\xi)\le 2,
\qquad
\TC(X)=7,
\qquad
\cat(X)=4.
\]
Thus this example shows that \(\TC(X,\xi)\) can take a value different from
those of \(\cat(X,\xi)\), \(\TC(X)\), and (unreduced) \(\cat(X)\).

%% file: productineq_sub.tex
\sloppy

\section{Product inequality}

We prove a product inequality for \(\TC(X,\xi)\).
It is the analogue of the standard product inequalities for topological
complexity and for the Lusternik--Schnirelmann category.

\begin{theorem}[Product inequality]
Let \(X\) and \(Y\) be finite connected cell complexes, and let
\(\omega_X\) and \(\omega_Y\) be closed 1-forms on \(X\) and \(Y\),
respectively.  Assume that
\[
\TC(X,[\omega_X])>0
\quad\text{and}\quad
\TC(Y,[\omega_Y])>0.
\]
Then
\[
\TC\bigl(X\times Y,[p_X^*\omega_X+p_Y^*\omega_Y]\bigr)
\le
\TC(X,[\omega_X])
+
\TC(Y,[\omega_Y])
-1,
\]
where
\[
p_X\colon X\times Y\to X,
\qquad
p_Y\colon X\times Y\to Y
\]
are the projections.
\end{theorem}

\begin{proof}
The proof follows the product formula for closed 1-forms in
\cite[Theorem~9]{FS01}.

Put
\[
r=\TC(X,[\omega_X]),
\qquad
s=\TC(Y,[\omega_Y]).
\]
Let
\[
\Omega_X=p_2^*\omega_X-p_1^*\omega_X
\]
be the closed 1-form on \(X\times X\) associated with \(\omega_X\), and let
\[
\Omega_Y=p_2^*\omega_Y-p_1^*\omega_Y
\]
be the corresponding closed 1-form on \(Y\times Y\).

We identify
\[
(X\times Y)\times (X\times Y)
\cong
(X\times X)\times (Y\times Y)
\]
by sending
\[
((x_0,y_0),(x_1,y_1))
\]
to
\[
((x_0,x_1),(y_0,y_1)).
\]
Writing \(\operatorname{pr}_X\) and \(\operatorname{pr}_Y\) for the
projections onto \(X\times X\) and \(Y\times Y\), respectively, the closed
1-form associated with \(p_X^*\omega_X+p_Y^*\omega_Y\) becomes
\[
\operatorname{pr}_X^*\Omega_X+\operatorname{pr}_Y^*\Omega_Y.
\]

Fix \(N>0\).  Choose \(N'>0\) sufficiently large.  By the definition of
\(\TC(X,[\omega_X])\) and \(\TC(Y,[\omega_Y])\), there exist open covers
\[
X\times X
=
F_X
\cup U_{X,1}\cup\cdots\cup U_{X,r}
\]
and
\[
Y\times Y
=
F_Y
\cup U_{Y,1}\cup\cdots\cup U_{Y,s}
\]
such that \(F_X\) is \(N'\)-movable with respect to \(\Omega_X\),
\(F_Y\) is \(N'\)-movable with respect to \(\Omega_Y\), and each
\(U_{X,i}\) and \(U_{Y,j}\) admits a continuous local section of the
corresponding endpoint fibration.

Set
\[
A_X=(X\times X)\setminus(U_{X,1}\cup\cdots\cup U_{X,r}),
\]
and
\[
A_Y=(Y\times Y)\setminus(U_{Y,1}\cup\cdots\cup U_{Y,s}).
\]
Then \(A_X\) and \(A_Y\) are closed, and
\[
A_X\subset F_X,
\qquad
A_Y\subset F_Y.
\]

By \cite[Lemma~10.1]{FS01}, applied to
\(A_X\subset F_X\) and \(A_Y\subset F_Y\), we may choose open neighborhoods
\[
V_X\supset A_X,
\qquad
V_Y\supset A_Y,
\]
and homotopies of \(X\times X\) and \(Y\times Y\), respectively, whose
restrictions to \(V_X\) and \(V_Y\) have the required \(N'\)-movability
properties with respect to \(\Omega_X\) and \(\Omega_Y\).

Define
\[
F=
\bigl(V_X\times (Y\times Y)\bigr)
\cup
\bigl((X\times X)\times V_Y\bigr).
\]
Using the homotopies on \(X\times X\) and \(Y\times Y\), we obtain a homotopy
of
\[
(X\times X)\times (Y\times Y).
\]
Restricted to \(F\), this homotopy makes \(F\) \(N\)-movable with respect to
\[
\operatorname{pr}_X^*\Omega_X+\operatorname{pr}_Y^*\Omega_Y,
\]
provided \(N'\) was chosen sufficiently large.

More explicitly, for every point
\[
q=((x_0,y_0),(x_1,y_1))\in F,
\]
one obtains paths
\[
\alpha_N^q,\qquad \beta_N^q
\]
in \(X\times Y\) such that, for some constant \(C>0\) independent of \(N\)
\cite[Lemma~10.1]{FS01},
\[
\int_{\beta_N^q|_{[0,t]}}
(p_X^*\omega_X+p_Y^*\omega_Y)
-
\int_{\alpha_N^q|_{[0,t]}}
(p_X^*\omega_X+p_Y^*\omega_Y)
\le C
\]
for all \(t\in[0,1]\), and
\[
\int_{\beta_N^q}
(p_X^*\omega_X+p_Y^*\omega_Y)
-
\int_{\alpha_N^q}
(p_X^*\omega_X+p_Y^*\omega_Y)
\le -N.
\]
Thus \(F\) is \(N\)-movable with respect to this form.

It remains to cover the complement of \(F\) by local planner domains.  Since
\(A_X\subset V_X\) and \(A_Y\subset V_Y\), we have
\[
(X\times X)\setminus V_X
\subset
(X\times X)\setminus A_X
=
U_{X,1}\cup\cdots\cup U_{X,r},
\]
and
\[
(Y\times Y)\setminus V_Y
\subset
(Y\times Y)\setminus A_Y
=
U_{Y,1}\cup\cdots\cup U_{Y,s}.
\]
Therefore,
\[
\begin{split}
&\bigl((X\times X)\times(Y\times Y)\bigr)\setminus F\\
&\qquad =
\bigl((X\times X)\setminus V_X\bigr)
\times
\bigl((Y\times Y)\setminus V_Y\bigr).
\end{split}
\]
Hence the complement of \(F\) is covered by the \(rs\) open sets
\[
U_{X,i}\times U_{Y,j},
\qquad
1\le i\le r,\quad 1\le j\le s.
\]
On each product \(U_{X,i}\times U_{Y,j}\), a local section is obtained by
taking the product of the local sections on \(X\) and \(Y\).

By the standard product-cover argument for topological complexity
\cite{Far01}, these \(rs\) product domains can be replaced by
\(r+s-1\) open sets, each admitting a continuous local section.
Therefore, for every \(N>0\), the space
\[
(X\times Y)\times (X\times Y)
\]
admits an open cover consisting of one \(N\)-movable set and \(r+s-1\)
local planner sets.  Hence
\[
\TC\bigl(X\times Y,[p_X^*\omega_X+p_Y^*\omega_Y]\bigr)
\le r+s-1.
\]
This proves the theorem.
\end{proof}

\begin{remark}
If either \(\TC(X,[\omega_X])=0\) or \(\TC(Y,[\omega_Y])=0\), then
\[
\TC\bigl(X\times Y,[p_X^*\omega_X+p_Y^*\omega_Y]\bigr)=0.
\]
Therefore, it is enough to consider the case where both factor invariants are
positive.
\end{remark}

%% file: navigation_finalfantasy.tex
\sloppy

\section{A navigation-function-type argument}

In this section we give a navigation-function-type argument for estimating
\(\TC(M,[\omega])\).
We first recall ordinary navigation functions.

\begin{definition}[Navigation function, {\cite[Definition~4.30]{Far04}}]
Let \(M\) be a closed smooth manifold.
A \(C^2\)-function
\[
F\colon M\times M\to\R
\]
is called a \emph{navigation function} if it is a non-negative
Morse--Bott function and
\[
F^{-1}(0)=\Delta_M.
\]
\end{definition}

The gradient flow of a navigation function decomposes \(M\times M\)
into stable manifolds of the critical submanifolds of \(F\).
There are several works in which motion planners are constructed using
gradient flows or geodesics; see, for example, \cite{BC01,MMP01,Mes01}.
To state the resulting estimate, we recall the following relative form
of topological complexity.

\begin{definition}[Subspace topological complexity,
{\cite[Definition~4.20]{Far04}}]
Let \(A\subset M\times M\).
The subspace topological complexity \(\TC_M(A)\) is the least integer
\(r\) such that \(A\) admits an open covering
\[
A=A_1\cup\cdots\cup A_r
\]
and, for each \(i\), there is a continuous map
\[
s_i\colon A_i\to PM
\]
satisfying
\[
\pi_M\circ s_i=\operatorname{id}_{A_i},
\]
where
\[
\pi_M\colon PM\to M\times M,
\qquad
\pi_M(\gamma)=(\gamma(0),\gamma(1)).
\]
\end{definition}

\begin{proposition}[{\cite[Proposition~4.24]{Far04}}]
\label{prop:subspace-TC-cover}
If
\[
A_1,\dots,A_k\subset M\times M
\]
are open subsets which cover \(M\times M\), then
\[
\TC(M)\le \TC_M(A_1)+\cdots+\TC_M(A_k).
\]
\end{proposition}

The invariant \(\TC_M(A)\) can be smaller than the subspace
Lusternik--Schnirelmann category
\(\cat_{M\times M}(A)\).
For instance,
\[
\TC_M(\Delta_M)=1,
\]
since the constant path gives a global planner on the diagonal.
On the other hand, one has
\[
\cat_{M\times M}(\Delta_M)=\cat(M).
\]
This explains why the navigation-function estimate can be sharper than
the usual LS-category estimate \cite{Costa01}.

\begin{theorem}[{\cite[Theorem~4.32]{Far04}}]
\label{thm:navigation-function-estimate}
Let
\[
F\colon M\times M\to\R
\]
be a navigation function.
Suppose that the positive critical values of \(F\) are
\[
0<c_1<\cdots<c_r,
\]
and put
\[
C_i=\operatorname{Crit}(F)\cap F^{-1}(c_i).
\]
Then
\[
\TC(M)\le 1+\sum_{i=1}^r \TC_M(C_i).
\]
\end{theorem}

Here the term \(1\) comes from the stable manifold of the diagonal
\(\Delta_M=F^{-1}(0)\), where the constant path gives a canonical
planner.
The summation terms come from the stable manifolds of the positive
critical levels, whose contributions are measured by
\(\TC_M(C_i)\).

We now use gradient flows of closed 1-forms to obtain a
navigation-function-type estimate of \(\TC(M,[\omega])\).

Let \((M,g)\) be a closed Riemannian manifold, and let \(\omega\) be a
smooth closed 1-form on \(M\).
Set
\[
Y=\{x\in M\mid \omega_x=0\},
\]
and assume that \(\omega\) is exact on a neighborhood \(U\) of \(Y\),
say
\begin{equation}
\label{navigation:conditionforomega}
\left.\omega\right|_U=dL
\end{equation}
for some smooth function \(L\colon U\to\R\).
Let
\[
V=-\omega^{\sharp_g}.
\]
Then
\[
V|_U=-\grad_g L.
\]
Let \(\phi_t\) denote the flow generated by \(V\).
We shall use the following standard consequence of the absence of
homoclinic chains.

\begin{definition}[Homoclinic chain]
Let \((M,g)\) be a closed Riemannian manifold, let \(\omega\) be a
closed 1-form on \(M\), and let \(\phi_t\) be the flow generated by
\[
V=-\omega^{\sharp_g}.
\]
Suppose that the zero set of \(\omega\) has connected components
\[
\operatorname{Zero}(\omega)=Y_1\cup\cdots\cup Y_k.
\]
A \emph{homoclinic chain} for \(\phi_t\) is a finite cyclically ordered
collection of orbits
\[
\gamma_1,\ldots,\gamma_l
\subset M\setminus\operatorname{Zero}(\omega),
\qquad
\gamma_{l+1}=\gamma_1,
\]
together with indices
\[
1\le i(j)\le k,
\qquad
1\le j\le l,
\]
such that
\[
\omega(\gamma_j)\cup\alpha(\gamma_{j+1})
\subset Y_{i(j)}
\]
for every \(j=1,\ldots,l\). Here \(\omega(\gamma)\) and
\(\alpha(\gamma)\) denote the forward and backward limit sets of the
orbit \(\gamma\), respectively.
\end{definition}

\begin{lemma}[{\cite[Section~4]{Lat01}}]
\label{lem:no-homoclinic-truncated-stable}
Let \((M,g)\) be a closed Riemannian manifold, and let \(\omega\) be a
closed 1-form satisfying
\eqref{navigation:conditionforomega}.
Let
\[
V=-\omega^{\sharp_g},
\]
and denote its flow by \(\phi_t\).
Assume that \(\phi_t\) has no homoclinic chains.
Write
\[
Y=Y_1\cup\cdots\cup Y_r
\]
for the decomposition of the zero set of \(\omega\) into connected
components.

Let \(W_i\subset U\) be prescribed neighborhoods of \(Y_i\).
Then, for every \(N>0\), one may choose closed neighborhoods
\[
B_{N,i}\subset W_i
\]
of \(Y_i\), which are isolating blocks, such that, with
\[
U_{N,i}=
\left\{
x\in M \ \middle|\
\begin{array}{l}
\text{there exists }t_x>0\text{ such that }
\phi_{t_x}(x)\in\Int(B_{N,i}),\\[2mm]
\displaystyle
\int_0^{t_x}\omega(\dot\phi_t(x))\,dt>-N
\end{array}
\right\},
\]
one has an open covering
\[
M=A_N\cup U_{N,1}\cup\cdots\cup U_{N,r},
\]
where
\[
A_N=
\left\{
x\in M
\ \middle|\
\text{there exists }t_x>0\text{ such that }
\int_0^{t_x}\omega(\dot\phi_t(x))\,dt=-N
\right\}.
\]
Moreover, for each \(i\), the flow gives a homotopy from \(U_{N,i}\)
into \(B_{N,i}\).
More precisely, on each \(U_{N,i}\) one may choose a continuous time
function
\[
T_{N,i}\colon U_{N,i}\to[0,\infty)
\]
such that
\[
\phi_{T_{N,i}(x)}(x)\in B_{N,i}
\]
for every \(x\in U_{N,i}\).
\end{lemma}

\begin{proof}
This is a consequence of Latschev's construction for flows admitting
Lyapunov 1-forms and no homoclinic chains; see
\cite[Section~4]{Lat01}.
The isolating blocks may be chosen inside the prescribed neighborhoods
\(W_i\).
\end{proof}

\begin{theorem}
\label{thm:gradient-flow-estimate}
Let \((M,g)\) be a closed connected Riemannian manifold, and let
\[
\xi\in H^1(M;\R).
\]
Let \(\Theta\) be a smooth closed 1-form on \(M\times M\) such that
\[
[\Theta]=p_2^*\xi-p_1^*\xi.
\]
Let \(\phi_t\) be the flow generated by
\[
V=-\Theta^{\sharp_{g\oplus g}}.
\]
Assume that \(\Theta\) is exact on a neighborhood of its zero set and
that \(\phi_t\) has no homoclinic chains.

Suppose that the zero set of \(\Theta\) has finitely many connected
ENR components, which we write as
\[
\operatorname{Zero}(\Theta)=Y_1\cup\cdots\cup Y_r.
\]
Then
\[
\TC(M,\xi)\le \sum_{i=1}^r\TC_M(Y_i).
\]
\end{theorem}

\begin{proof}
For each \(i\), put
\[
m_i=\TC_M(Y_i).
\]
Choose a relatively open covering
\[
Y_i=A_{i,1}\cup\cdots\cup A_{i,m_i}
\]
admitting local motion planners.

Since \(Y_i\) is a compact ENR, there exists an open neighborhood
\(W_i\) of \(Y_i\), contained in the neighborhood on which \(\Theta\)
is exact, together with a deformation retraction
\[
r_i\colon W_i\to Y_i.
\]
For \(1\le j\le m_i\), put
\[
O_{i,j}=r_i^{-1}(A_{i,j}).
\]
Then
\[
W_i=O_{i,1}\cup\cdots\cup O_{i,m_i}.
\]
The local motion planners on \(A_{i,j}\) transfer, along the deformation
retraction of \(W_i\) onto \(Y_i\), to local motion planners on
\(O_{i,j}\).

Let \(N>0\).
Applying Lemma~\ref{lem:no-homoclinic-truncated-stable} with the
prescribed neighborhoods \(W_i\), we obtain closed neighborhoods
\[
B_{N,i}\subset W_i
\]
and an open covering
\[
M\times M
=
A_N\cup V_{N,1}\cup\cdots\cup V_{N,r},
\]
where \(A_N\) is movable with integral threshold \(-N\).
Moreover, for each \(i\), the flow gives a homotopy
\[
K_i\colon V_{N,i}\times[0,1]\to M\times M
\]
from the inclusion of \(V_{N,i}\) to a map
\[
\rho_i\colon V_{N,i}\to B_{N,i}\subset W_i.
\]

For \(1\le j\le m_i\), set
\[
V_{N,i,j}=\rho_i^{-1}(O_{i,j}).
\]
These sets form an open cover of \(V_{N,i}\).
By the homotopy lifting property, the local motion planner on
\(O_{i,j}\) transfers along \(K_i\) to a local motion planner on
\(V_{N,i,j}\).

Consequently,
\[
M\times M
=
A_N\cup
\bigcup_{i=1}^r\bigcup_{j=1}^{m_i}V_{N,i,j}
\]
is a covering of the type appearing in the definition of
\(\TC(M,\xi)\).
Therefore
\[
\TC(M,\xi)\le \sum_{i=1}^r m_i
=
\sum_{i=1}^r\TC_M(Y_i).
\]
Since \(N>0\) was arbitrary, the proof is complete.
\end{proof}

\begin{example}[A circle bundle over \(\Sigma_2\)]
Let
\[
S^1\hookrightarrow E\xrightarrow{\pi}\Sigma_2
\]
be a smooth circle bundle over the closed oriented surface of genus
\(2\).
Let
\[
0\ne \xi_\Sigma\in H^1(\Sigma_2;\Z).
\]
By a result of Sch\"utz \cite[Proposition~4.1]{Sch01}, one can choose a
closed 1-form \(\alpha\) representing \(\xi_\Sigma\) such that
\[
\operatorname{Zero}(\alpha)=\{q\}
\]
and such that \(\alpha\) admits a negative gradient-like flow with no
homoclinic chains.
Put
\[
\omega=\pi^*\alpha.
\]
Then
\[
[\omega]=\pi^*\xi_\Sigma\in H^1(E;\R)
\]
and
\[
\operatorname{Zero}(\omega)=\pi^{-1}(q)\cong S^1.
\]

Put
\[
\Theta=p_2^*\omega-p_1^*\omega
\]
on \(E\times E\).
Then
\[
[\Theta]
=
p_2^*[\omega]-p_1^*[\omega],
\]
and
\[
\operatorname{Zero}(\Theta)
=
\pi^{-1}(q)\times\pi^{-1}(q).
\]
Thus the unique zero component of \(\Theta\) is
\[
C=\pi^{-1}(q)\times\pi^{-1}(q)\subset E\times E.
\]

Choose a connection on the circle bundle and a compatible bundle metric on \(E\). Then the negative gradient flow of \(\omega\) is the horizontal lift of that of \(\alpha\). With respect to the product metric on \(E\times E\), the negative gradient flow of \(\Theta\) is the product of the reversed lifted flow in the first factor and the lifted flow in the second factor. If this flow admitted a homoclinic chain, the projection of a nonconstant coordinate to \(\Sigma_2\) would give such a chain for the flow of \(\alpha\), which is impossible.

Moreover, \(\Theta\) is exact on a neighborhood of \(C\).
Indeed, if
\[
\left.\alpha\right|_D=d\ell
\]
on a neighborhood \(D\) of \(q\), then on
\[
\pi^{-1}(D)\times\pi^{-1}(D)
\]
one has
\[
\Theta
=
d\left(
p_2^*(\ell\circ\pi)-p_1^*(\ell\circ\pi)
\right).
\]

Therefore Theorem~\ref{thm:gradient-flow-estimate} gives
\[
\TC(E,[\omega])\le \TC_E(C).
\]

For every pair in \(C\), both endpoints lie in the fiber
\[
\pi^{-1}(q)\cong S^1.
\]
Thus a motion planner for the fiber gives a motion planner in \(E\),
and hence
\[
\TC_E(C)\le \TC(S^1).
\]
Since
\[
\TC(S^1)=2,
\]
we obtain
\[
\TC(E,[\omega])\le 2.
\]
\end{example}

\begin{remark}
If the circle bundle is trivial, so that
\[
E\cong \Sigma_2\times S^1,
\]
then the product inequality gives
\[
\begin{aligned}
\TC(E,[\omega])
&\le
\TC(\Sigma_2,[\alpha])+\TC(S^1,0)-1\\
&=
1+2-1\\
&=2.
\end{aligned}
\]
Thus, in the trivial case, the preceding example recovers the same
upper bound as the product inequality.

For a nontrivial circle bundle, however, \(E\) does not in general
decompose as \(\Sigma_2\times S^1\), and the product inequality cannot
be applied directly. The preceding argument nevertheless gives
\[
\TC(E,[\omega])\le 2.
\]
Hence the navigation-function-type estimate extends the product
estimate to nontrivial circle bundles.
\end{remark}